\documentclass[12pt]{article}
\usepackage{a4wide}
\usepackage{inputenc}
\usepackage[a4paper, margin=2.3cm]{geometry}
\usepackage{amsmath,amssymb,amsthm}
\usepackage{color}
\usepackage{parskip}
\usepackage{hyperref}
\usepackage{parskip}
\usepackage{setspace}
\usepackage{graphicx}
\usepackage{cases}
 
\newtheorem{theorem}{Theorem}[section]

\newtheorem{lemma}[theorem]{Lemma}
\newtheorem{proposition}[theorem]{Proposition}

\newtheorem*{definition*}{Definition}

\def\AA{\mathcal{A}}
\def \BB{\mathcal{B}} 
\def\EE{\mathcal{E}}
\def\FF{\mathcal{F}}

\begin{document}

\title{A point-plane incidence theorem in matrix rings} 

\author{ Nguyen Van The \thanks{University of Science, Vietnam National University - Hanoi, Email: nguyenvanthe@hus.edu.vn} \and Le Anh Vinh\thanks{Vietnam National University - Hanoi, Email: vinhla@vnu.edu.vn. Vietnam Institute of Educational Sciences. Email: vinhle@vnies.edu.vn}}
\date{}
\maketitle  
\begin{abstract}
In this paper, we study a point-hyper plane incidence theorem in matrix rings, which generalizes all previous works in literature of this direction. 
\end{abstract}

\section{Introduction}
Let $\mathbb{F}_q$ be a finite field of order $q$ where $q$ is an odd prime power. Let $M_n(\mathbb{F}_q)$ be the set of $n \times n$ matrices with entries in $\mathbb{F}_q$ and $GL_n(\mathbb{F}_q)$ be the set of invertible matrices in $M_n(\mathbb{F}_q)$. For $A_1,A_2,\dots,A_d, B \in M_n(\mathbb{F}_q),$ we define a hyper-plane in $M_n(\mathbb{F}_q)^{d+1} = M_n(\mathbb{F}_q) \times \dots \times M_n(\mathbb{F}_q)$ as the set of points $(X_1,\dots,X_d,Y) \in M_n(\mathbb{F}_q)^{d+1}$ satisfying  
\begin{equation}\label{eq.main}
    A_1X_1 + \dots A_dX_d + B = Y.
\end{equation} 
Note that, this definition is a formal definition, however, it is associated to an $d$-dimensional affine plane over the module $M_n(\mathbb{F}_q)^{d+1}$. When $d = 1$, we say the set of points satisfying \eqref{eq.main} is a line in $M_n(\mathbb{F}_q)^2$ formed by $A_1$ and $B$. In the first version of this note, we prove the following point-line incidence with $d=1$ and $n = 2$. 
\begin{theorem}
Let $\mathcal{P}$ be a set of points and $\mathcal{L}$ be a set of lines in $M_2(\mathbb{F}_q) \times M_2(\mathbb{F}_q)$, then we have
\[ I(\mathcal{P},\mathcal{L}) \leq \dfrac{|\mathcal{P}||\mathcal{L}|}{q^4} + \sqrt{2} q^{7/2}\sqrt{|\mathcal{P}||\mathcal{L}|}.\]
\end{theorem}
During revision process, Xie and Ge \cite{Xie} extended this results by considering the equation $AX + BY = C + D$ where $A, B, C, D, X, Y \in M_n(\mathbb{F}_q)$, which also generalizes the work of Mohammadi, Pham and Wang \cite{Moh} for $n = 2$. Some similar results could be found in \cite{Kara, The}, in which they consider the sum-product equation. 

In this note, we prove a generalization of all above mentioned results in \cite{Kara, Moh, The, Xie}. Particularly, we consider the equation \eqref{eq.main} for all $d \geq 1$ and $n \geq 2$. 

\begin{theorem}\label{theo.main}
For $d \geq 1, n \geq 2,$ let $\mathcal{A}_1,\dots,\mathcal{A}_d,\mathcal{B}_1,\dots,\mathcal{B}_d, \mathcal{E}, \mathcal{F} \subset M_n(\mathbb{F}_q)$ and $N$ be the number of solutions to the sum-product equation
\[ A_1B_1 + A_2B_2+\dots+A_dB_d = E + F, \quad A_i \in \mathcal{A}_i, B_i \in \mathcal{B}_i, E \in \mathcal{E}, F \in \mathcal{F}.\]
Then we have 
\[ \left| N - \dfrac{|\EE||\FF||\AA_1||\BB_1|\dots |\AA_n||\BB_n|}{q^{n^2}}\right| \ll q^{dn^2 - (d-1)n/2 - 1/2} \sqrt{|\EE||\FF|\prod_{i=1}^n|\AA_i||\BB_i|}.\]
\end{theorem}

Theorem \ref{theo.main} will be proved via spectral graph theory, in particular, the expander mixing lemma of regular directed graph given by Vu \cite{Vu}. Firstly, we recall some definitions from graph theory as the following subsection, which is extracted from Section 4 in \cite{The}.

\section{Tools from spectral graph theory}
Let $G$ be a directed graph (digraph) on $n$ vertices where the in-degree and out-degree of each vertex are both $d.$\\[12pt]
Let $A_G$ be the adjacency matrix of $G$, i.e., $a_{ij} = 1$ if there is a directed edge from $i$ to $j$ and zero otherwise. Suppose that $\lambda_1 = d, \lambda_2,...,\lambda_n$ are the eigenvalues of $A_G.$ These eigenvalues can be complex, so we cannot order them, but it is known that $|\lambda_i| \le d$ for all $1 \le i \le n.$ Define $\lambda(G):= \max_{|\lambda_i| \neq d } |\lambda_i|.$ This value is called the second largest eigenvalue of $A_G.$ We say that the $n \times n$ matrix $A$ is normal if $A^tA=AA^t$ where $A^t$ is the transpose of $A.$ The graph $G$ is normal if $A_G$ is normal. There is a simple way to check whenever $G$ is normal or not. Indeed, for any two vertices $x$ and $y,$ let $N^+(x,y)$ be the set of vertices $z$ such that $\overrightarrow{xz},\overrightarrow{yz}$ are edges, and $N^-(x,y)$ be the set of vertices $z$ such that $\overrightarrow{zx},\overrightarrow{zy}$ are edges. By a direct computation, we have $A_G$ is normal if and only if $|N^+(x,y)| = |N^-(x,y)|$ for any two vertices $x$ and $y.$ 

\medskip\noindent A digraph $G$ is called an $(n,d,\lambda)-digraph$ if $G$ has $n$ vertices, the in-degree and out-degree of each vertex are both $d,$ and $\lambda(G) \le \lambda.$ Let $G$ be an $(n,d,\lambda)-digraph.$ We have the following expander mixing lemma given by Vu \cite{Vu}.

\begin{lemma}[Expander Mixing Lemma, \cite{Vu}]\label{Expander.Mixing.Lemma}
Let $G = (V,E)$ be an $(n,d,\lambda)-digraph.$ For any two sets $B,C \subset V,$ we have
\[ \left| e(B,C) - \frac{d}{n}|B||C|\right| \le \lambda\sqrt{|B||C|}\]
where $e(B,C)$ be the number of ordered pairs $(u,w)$ such that $u\in B, w \in C,$ and $\overrightarrow{uw} \in e(G).$
\end{lemma}
As $d|B||C|/n$ is the expected number of edges from $B$ to $C$, the above lemma gives us a bound for the gap between this number and the number of edges between $B$ and $C$.

\section{Proof of Theorem \ref{theo.main}} 

To prove Theorem \ref{theo.main}, we define the sum-product digraph $G = (V,E)$ with the vertex set 
$$ V = M_n(\mathbb{F}_q)^{d+1} = M_n(\mathbb{F}_q) \times M_n(\mathbb{F}_q) \times \dots \times M_n(\mathbb{F}_q)$$.
There exists a directed edge from $(A_1,\dots,A_d,E)$ to $(B_1,\dots,B_d,F)$ if 
\[ A_1B_1 + A_2B_2 + \dots + A_dB_d = E + F.\]
Theorem \ref{theo.main} will be directly followed from Lemma \ref{Expander.Mixing.Lemma} and the following proposition. 

\begin{proposition}\label{digraph.theorem}
The sum-product digraph $G = (V,E)$ is an $\left(q^{(d+1)n^2}, q^{dn^2},Cq^{dn^2 - (d-1)n/2 - 1/2}\right)-digraph$ for some positive constant $C$.
\end{proposition}

To prove Proposition \ref{digraph.theorem}, we need some basic facts from linear algebra via the following lemmas. 

\begin{lemma}\label{rank.k.lemma}
Denote $M_{n \times t}(\mathbb{F}_q)$ be the set of $n \times t$ matrices with entries restricted in $\mathbb{F}_q$ with $n < t$. Then the number of matrices of rank $k \leq n$ is less than
\[ \binom{t}{k} \left(q^{n}-1\right)\left(q^n - q\right) \dots \left(q^n - q^{k-1}\right) q^{(t-k)k} \leq C_{t,k} q^{nk+k(t-k)}\]
for some positive constant $C_{t,k}$. 
\end{lemma}

\begin{proof}
It is observed that the number of $k$ independent column vectors $\mathbf{c}_1,\dots,\mathbf{c}_k$ in $\mathbb{F}_q^n$ is 
\[ \left(q^n-1\right)\left(q^n - q\right) \dots \left(q^n-q^{k-1}\right).\]
For each matrix $A \in M_{n \times t}(\mathbb{F}_q)$ of rank $k$, there exist $k$ independent column vectors $\mathbf{c}_1,\dots,\mathbf{c}_k$ in $A$. Furthermore, the $t-k$ other columns can be written as a linear combination of $\left\{\mathbf{c}_{i}\right\}_{i=1}^k,$ each has $q^k$ possibilities. We can choose $k$ of $t$ column vectors, which are linearly independent. Since two different ways can be just one matrix, the number of matrices of rank $k \leq n$ is bounded by
\begin{align*}
\binom{t}{k} \left(q^n-1\right)\left(q^n-q^2\right)\dots\left(q^n-q^{k-1}\right)q^{k(t-k)} \leq C q^{nk+k(t-k)}
\end{align*}
for some positive constant $C$.
\end{proof}

\begin{lemma}\label{lemma.system MZ = C}
For $n \geq m \geq k \geq 0$, let $\mathcal{T}_{m,k}$ be the number of pair $(M,C)$ with $M \in M_{n \times t}\left(\mathbb{F}_q\right), C \in M_n(\mathbb{F}_q)$ satisfying $\mathrm{rank}(M) = m, \mathrm{rank}(C)=k$, and the system $MZ = C$ has solution. We have
\[ 
\mathcal{T}_{m,k} \ll  q^{nm+m(t-m)+mk+k(n-k)}
\] 
\end{lemma}
\begin{proof}
For each matrix $M \in M_{n \times t}$ of rank $m$, we will bound the number of matrices $C \in M_n(\mathbb{F}_q)$ of rank $k$ that satisfy the equation $MZ = C$ has solution. It follows from Kronecker-Capelli theorem that the equation $MZ  = C$ has solution if and only if $\mathrm{rank}(M) = \mathrm{rank}(\overline{M})$ where $\overline{M} = \left( M \, C \right)$ is the matrix obtained by adding $C$ to $M$ on the right. 

Let $\mathbf{r}_{M_1},\dots, \mathbf{r}_{M_n}$ be row vectors of $M$ and $\mathbf{r}_{C_1},\dots,\mathbf{r}_{C_n}$ be row vectors of $C$, respectively. Without loss of generality, we assume that $\mathbf{r}_{M_1},\dots,\mathbf{r}_{M_k}$ are linearly independent and 
\[ \mathbf{r}_{M_i} = \alpha_{1_i} \mathbf{r}_{M_1} +\alpha_{2_i} \mathbf{r}_{M_2} + \dots +\alpha_{m_i} \mathbf{r}_{M_m}, \, i =m+1, m+2, \dots,n \]
for some $\alpha_{j_i} \in \mathbb{F}_q$ for all $j = 1,\dots,m$. Since $\mathrm{rank}(M) = \mathrm{rank}\left(\overline{M}\right)$, we have 
\[ \mathbf{r}_{C_i} = \alpha_{1_i} \mathbf{r}_{C_1} +\alpha_{2_i} \mathbf{r}_{C_2} + \dots +\alpha_{m_i} \mathbf{r}_{C_m}, \, i =m+1, m+2, \dots,n. \]
This means the $n-m$ later row vectors of $C$ are uniquely determined by first $m$ row vectors $\mathbf{r}_{C_1},\dots,\mathbf{r}_{C_m}$ for a given matrix $M$. Therefore, we only need to count number of $m$ row vectors $\mathbf{r}_{c_1},\dots,\mathbf{r}_{c_m}$ satisfying $\mathrm{rank}\left(\mathbf{r}_{c_1},\dots,\mathbf{r}_{c_m}\right) = k$, or equivalently, the number of $m \times n$ matrices of rank $k$. Hence, it follows from Lemma \ref{rank.k.lemma} that the number of matrices $C$ satisfying $MZ = C$ has solution is bounded by $dq^{mk+k(n-k)}$ for some positive constant $d$. 

Again, applying Lemma \ref{rank.k.lemma}, there are at most $dq^{nm+m(t-m)}$ matrices $M$ of rank $m$ in $M_{n \times t}$. Thus, we obtain 
\[ \mathcal{T}_{m,k} \ll q^{nm+m(t-m)+mk+k(n-k)},\]
which completes the proof of Lemma \ref{lemma.system MZ = C}.
\end{proof}

We are now ready to prove Proposition \ref{digraph.theorem}. 

\begin{proof}[\bf Proof of Proposition \ref{digraph.theorem}]
It is obvious that the order of $G$ is $q^{(d+1)n^2},$ because $\left| M_n(\mathbb{F}_q)\right|= q^{n^2}$ and so $\left|M_n(\mathbb{F}_q)
^{d+1}\right|=q^{(d+1)n^2}.$ Next, we observe that $G$ is a regular digraph of in-degree and out-degree $q^{dn^2}.$ Indeed, for any vertex $(A_1,\dots,A_d,E) \in V,$ if we choose each $d$-tuple $(B_1,\dots,B_d) \in M_n(\mathbb{F}_q)^d,$ there exists a unique $F = A_1B_1 + \dots +A_dB_d - E$ such that 
\[ A_1B_1 + \dots + A_dB_d = E + F.\]
Hence, the out-degree of any vertex in $G$ is $\left|M_n(\mathbb{F}_q)^d\right|,$ which is $q^{dn^2}.$ The same holds for the in-degree of each vertex. Therefore, to prove Proposition \ref{digraph.theorem}, we only need to bound the second largest eigenvalue of $G$. 

To this end, we first need to show that $G$ is a normal digraph. Let $A_G$ be the adjacency matrix of $G$. It is known that if $A_G$ is a normal matrix and $\beta$ is an eigenvalue of $A_G,$ then the complex conjugate $\overline{\beta}$ is an eigenvalue of $A_G^t.$ Hence, $|\beta|^2$ is an eigenvalue of $A_GA_G^t$ and $A_G^tA_G.$ In other words, in order to bound $\beta,$ it is enough to bound the second largest eigenvalue of $A_GA_G^t.$ 

Firstly, we will show that $G$ is a normal graph. Let $(A_1,A_2,\dots,A_d,E)$ and $(A'_1,A'_2,\dots,A'_d,E')$ be two different vertices, we now count the of the neighbors $(B_1,\dots,B_d,F)$ such that there are directed edges from $(A_1,\dots,A_d,E)$ and $(A'_1,\dots,A'_d,E')$ to $(B_1,\dots,B_d,F).$ This number is $N^+((A_1,\dots,A_d,E),(A'_1,\dots,A'_d,E')).$ We have 
\begin{align} \label{eq.system}
A_1B_1+ \dots + A_dB_d = E + F \text{  and  } 
A'_1B_1 + \dots + A'_dB_d = E' + F,
\end{align}
which implies 
\begin{align}\label{equation 1}
 (A_1-A'_1) B_1 + \dots + (A_d - A'_d)B_d = A - A'.
\end{align}
Note that if we fix a solution $(B_1,\dots,B_d)$ to the equation \eqref{equation 1}, then $F$ in \eqref{eq.system} is uniquely determined. Let $M = \left( A_1 - A'_1 \, \,\, A_1 - A'_1 \,\, \dots \,\,  A_d - A'_d \right)$, $X = \left( B_1 \,\,\, B_2 \,\,\dots \,\,B_d \right)$ and $Y = E - E'$, then the equation can be rewritten as the following matrix equation.
\begin{equation}\label{matrix.equa}
MX = Y
\end{equation}
with $M \in M_{n \times dn}(\mathbb{F}_q), Y \in M_n(\mathbb{F}_q)$ and $X \in M_{dn \times n}(\mathbb{F}_q)$. We now fall into the following cases. 
\begin{itemize}
\item \textbf{Case 1.} If $\mathrm{rank}(M) = n$,  then there exists unique $X$ such that $MX = Y$. Thus the system \eqref{eq.system} has only one solution in this case.

\item \textbf{Case 2.} If $\mathrm{rank}(M) = m$ and $\mathrm{rank}(Y) = k$ with $m < k \leq n$, then the equation $MX = Y$ has no solution. 

\item \textbf{Case 3.} If $\mathrm{rank}(M) = m$ and $\mathrm{rank}(Y) = k$ with $n > m \geq k$,  we need to further consider different situations as follows.
	\begin{itemize}
	\item[-] \textbf{Case 3.1.} If $\mathrm{rank}(M) = \mathrm{rank}(Y) = 0$, then $E = E', A_i = A'_i$ for all $i = 1,\dots,d$, which contracts with our assumption that $(A_1,A_2,\dots,A_d,E)$ and $(A'_1,A'_2,\dots,A'_d,E')$ are two different vertices. Thus, we can rule out this case. 
	
	\item[-] \textbf{Case 3.2.} If $\mathrm{rank}\left(\overline{M}\right) > m$ where $\overline{M} = \left( M \,\, Y\right)$, it follows from Kronecker-Capelli theorem that the equation $MX = E$ has no solution.
	
	\item[-] \textbf{Case 3.3.} Suppose that $\mathrm{rank}\left(\overline{M}\right) = m$ where $\overline{M} = \left( M \,\, Y\right)$, we have known that the equation $MX=Y$ has solution by the Kronecker Capelli theorem. Moreover, the number of its solutions is equal to the number of solutions of $MX=0.$ Now, we are ready to count the number of solutions of $MX=0$. 
	
 Put $X = [x_1 \,\, x_2 \, \dots \,x_n]$ for some column vectors $x_1,x_2,\dots,x_n \in \mathbb{F}_q^{dn},$ we just need to estimate the number of solutions of $Mx_1=0$ because the number of solutions of $MX=0$ is equal to $n^{th}$ power of the number of solutions of $Mx_1=0.$ It is known that the set $L$ of all solutions of $Ax_1=0$ is a vector subspace of $\mathbb{F}_q^{dn}$ and has the dimension $\dim L = dn - \mathrm{rank}(M) = dn-m.$ Hence, we have 
\[ |L|=q^{dn-m}.\] 
Therefore, the equation $MX=0$ has $q^{n(dn-m)}$ solutions if this equation has solutions. 	
	\end{itemize}
\end{itemize}
Since the same argument works for the case of $N^-((A_1,\dots,A_d,E),(A'_1,\dots,A'_d,E')),$ we obtain the same value for $N^-((A_1,\dots,A_d,e),(A'_1,\dots,A'_d,E')).$ In short, $A_G$ is normal.

As we discussed above, in order to bound the second largest eigenvalue of $M,$ it is enough to bound the second largest value of $MM^t.$ Note that each entry of $A_GA_G^t$ can be interpreted as counting the number of common outgoing neighbors between two vertices. Based on previous calculations, we have
\begin{align*}
 A_GA_G^t = \left(q^{dn^2}-1\right) I + J  - \sum_{0 \leq m < k \leq n} F_{m,k} - \sum_{\substack{0 \leq k \leq m < n : \\ (m,k) \neq (0,0)}}H_{m,k} + \left(q^{n(n-m)}-1\right) \sum_{\substack{0 \leq k \leq m < n : \\ (m,k) \neq (0,0)}}E_{m,k} 
\end{align*}
where $I$ is the identity matrix, $J$ denotes the all-one matrix and the others defined as follows. 

$F_{m,k}$ is the adjacency matrix of the graph $\mathcal{G}_{m,k}, 0 \leq m < k \leq n $  with the vertex set $V(F_{m,k}) = M_{n}\left(\mathbb{F}_q\right)^{d+1}$ and there is an edge between $(A_1,\dots,A_d,E)$ and $(A'_1,\dots,A'_d,E')$ if 
\[ 
\mathrm{rank}\left(A_1 - A'_1 \,\,\dots \,\, A_d - A'_d\right) = m \text{  and  } \mathrm{rank}\left(E - E'\right) = k.
\]
$H_{m,k}$ is the adjacency matrix of the graph $\mathcal{G}_{m,k}, 0 \leq k \leq m < n, (m,k) \neq (0,0)$ with the vertex set $V(H_{m,k}) = M_{n}\left(\mathbb{F}_q\right)^{d+1}$ and there is an edge between $(A_1,\dots,A_d,E)$ and $(A'_1,\dots,A'_d,E')$ if $\mathrm{rank}\left(A_1 - A'_1 \,\,\dots \,\, A_d - A'_d\right) = m, \, \mathrm{rank}\left(E - E'\right) = k$ and \[\mathrm{rank}\left(A_1-A'_1 \,\, \dots \,\, A_d - A'_d \,\,\, E - E'\right) > m.
\]
$E_{m,k}$ is the adjacency matrix of the graph $\mathcal{GP}_{m,k}, 0 \leq k \leq m < n, (m,k) \neq (0,0)$ with the vertex set $V(E_{m,k}) = M_{n}\left(\mathbb{F}_q\right)^{d+1}$ and there is an edge between $(A_1,\dots,A_d,E)$ and $(A'_1,\dots,A'_d,E')$ if $\mathrm{rank}\left(A_1 - A'_1 \,\,\dots \,\, A_d - A'_d\right) = m, \, \mathrm{rank}\left(E - E'\right) = k$ and \[\mathrm{rank}\left(A_1-A'_1 \,\, \dots \,\, A_d - A'_d \,\,\, E - E'\right) = m.\]

Using the Lemma \ref{rank.k.lemma}, one can easily to check that for any $0 \le k \leq n, 0 \leq m < n$ and  $(m,k) \neq (0,0),$ the graph $\mathcal{G}_{mk}$ is $d_{mk}-\text{regular}$ for some $d_{mk}$ where 
\[ d_{mk} \ll q^{nm+m(dn-m)}q^{nk+k(n-k)} = q^{(d+1)nm + 2nk - m^2 - k^2} \leq q^{(d+1)n^2 - (d-1)n - 1}.\] 

For the graph $\mathcal{GP}_{m,k},0 \leq k \leq m < n, (m,k) \neq (0,0),$ it follows directly from Lemma \ref{lemma.system MZ = C} that $\mathcal{G}_{mk}$ is $y_{mk}-\text{regular}$ for some $y_{mk}$ where 
\[ y_{m,k} \ll  q^{nm+m(dn-m) + mk + k(n-k)}. \]

Suppose $\lambda_2$ is the second largest eigenvalue of $A_G$ and $\overrightarrow{v_2}$ is the corresponding eigenvector. Since $G$ is a regular graph, we have $J \cdot \overrightarrow{v_2}=0.$ (Indeed, since $G$ is regular, it always has $(1,1,\dots,1)$ as an eigenvector with eigenvalue being its regular-degree. Moreover, since the graph $G$ is connected, this eigenvalue has multiplicity one. Thus any other eigenvectors will be orthogonal to $(1,1,\dots,1)$ which in turns gives us $J\cdot \overrightarrow{v_2}=0).$ Since $A_GA_G^t\overrightarrow{v_2}=|\lambda_2|^2\overrightarrow{v_2},$ we get
\begin{equation}\label{lambda.equation}
|\lambda_2|^2 \overrightarrow{v_2} = \left[\left(q^{dn^2}-1\right) - \sum_{0 \leq m < k \leq n} F_{m,k} - \sum_{\substack{0 \leq k \leq m < n : \\ (m,k) \neq (0,0)}}H_{m,k} + \left(q^{n(n-m)}-1\right) \sum_{\substack{0 \leq k \leq m < n : \\ (m,k) \neq (0,0)}}E_{m,k} \right] \overrightarrow{v_2}.
\end{equation}
One easily to check that $y_{m,k} q^{n(dn-m)} \ll q^{2dn^2 - (d-1)n - 1}$ for all $0 \leq k \leq m < n, (m,k) \neq 0$. Therefore, all previous calculations and the equation \eqref{lambda.equation} give us
\[ |\lambda_2|^2 \ll q^{2dn^2 - (d-1)n - 1},\]
since eigenvalues of a sum of matrices are bounded by the sum of largest eigenvalue of the summands. We complete the proof of Proposition \ref{digraph.theorem}. 
\end{proof}


\end{document}